\documentclass[11pt]{amsart}
\usepackage{amsmath, amsthm, amscd, amsfonts,amssymb, graphicx}
\usepackage{accents}
\usepackage{xcolor}
\usepackage{hyperref}

\newlength{\dhatheight}

\newcommand{\bea}{\begin{eqnarray*}}
\newcommand{\eea}{\end{eqnarray*}}

\newcommand{\beq}{\begin{equation}}
\newcommand{\eeq}{\end{equation}}
\newcommand{\bfomega}{\mbox{\boldmath $\omega$ \unboldmath} \hskip -0.05 true in}

\newcommand{\bom}{\bfomega}

  \newcommand{\ii}{\mathrm{i}}
 
\newtheorem{theorem}{Theorem}[section]

\theoremstyle{definition}

\newtheorem{proposition}[theorem]{Proposition}
\newtheorem{corollary}[theorem]{Corollary}

\theoremstyle{remark}
\newtheorem{remark}[theorem]{Remark}

\numberwithin{equation}{section}

\begin{document}

\title[Fourier-Zernike Series of Convolutions on Disks]{Fourier-Zernike Series of Convolutions on Disks}

\author[A. Ghaani Farashahi]{Arash Ghaani Farashahi$^*$}
\address{Laboratory for Computational Sensing and Robotics (LCSR), Whiting School of Engineering, Johns Hopkins University, Baltimore, Maryland, United States.}
\email{arash.ghaanifarashahi@jhu.edu}
\email{ghaanifarashahi@outlook.com}

\author[G.S. Chirikjian]{Gregory S. Chirikjian}
\address{Laboratory for Computational Sensing and Robotics (LCSR), Whiting School of Engineering, Johns Hopkins University, Baltimore, Maryland, United States.}
\email{gregc@jhu.edu}
\email{gchirik@gmail.com}

\date{\today}

\thanks{$^*$Corresponding author}
\thanks{E-mail addresses: arash.ghaanifarashahi@jhu.edu (Arash Ghaani Farashahi) and gregc@jhu.edu (Gregory S. Chirikjian)}

\begin{abstract}
This paper presents a systematic study for analytic aspects of Fourier-Zernike series of convolutions of functions supported on disks. We then investigate different aspects of the presented theory in the cases of zero-padded functions. 
\end{abstract}

\maketitle

\section{{\bf Introduction}}

The mathematical theory of convolution function algebras has significant roles 
in classical harmonic analysis, representation theory, functional analysis, and operator theory, see \cite{AGHF.CJM, AGHF.JAuMS, AGHF.IJM.2015, Kisil.Adv, Kisil.BJMA, 50} and references therein. Over the last decades, some new aspects of convolution function algebras have achieved significant popularity in modern harmonic analysis areas such as coorbit theory (including Gabor and wavelet analysis) \cite{Fei0, Fei1, Fei2, Fei.Gro1, Fei.Gro2} and recent applications in computational science and engineering \cite{PIb2, PIb4, PI6, Kya.PI.2000, Kya.PI.1999}.

In many applications in engineering, convolutions and
correlations of functions on Euclidean spaces are required. This includes template matching in image processing for pattern recognition, and protein docking \cite{Kim.Kim, 19Venkatraman2009, 19Venkatraman2009a}, and
characterizing how error probabilities propagate \cite{PI5}. In some applications, the goal is not to recover the values of convolved functions, but rather their support, which is the Minkowski sum of the supports of the two functions being convolved \cite{Kavraki}.
In most of these applications the functions of interest
take non-negative values, and as such can be normalized and treated as probability density functions (pdfs).

Usually two approaches are taken to computing convolutions of pdfs on Euclidean space. First, if the
functions are compactly supported, then their supports are enclosed in a solid cube with dimensions at least twice the size of the support of the functions, and  periodic versions of the functions are constructed.
In this way, convolution of these periodic functions on the $d$-torus can be used to replace convolution on $d$-dimensional Euclidean space. The benefit of this is that the spectrum is discretized and Fast Fourier Transform (FFT) methods can be used to compute the convolutions. This approach is computationally attractive, but in this periodization procedure the natural invariance of integration
on Euclidean space under rotation transformations is lost when moving to the torus. This can be a significant issue in rotation matching problems.

A second approach is to take the original compactly supported functions and replace them with functions on Euclidean space that have rapidly decaying
tails, but for which convolutions can be computed in closed form. For example, replacing each of the given functions with a sum of Gaussian distributions allows the convolution of the given functions to be computed as 
a sum of convolution of Gaussians, which have simple closed-form expressions as Gaussians. The problem with this approach is that the resulting functions are not compactly supported. Moreover, if $N$ Gaussians are used to describe each input function, then $N^2$ Gaussians result after the convolution.

An altogether different approach is explored here. Rather than periodizing the given functions, or extending their support to the whole of Euclidean space,
we consider functions that are supported on disks in the
plane (and by natural extension, to balls in higher dimensional Euclidean spaces). The basic idea is that
in polar coordinates each function is expanded in an orthonormal basis
consisting of Zernike polynomials in the radial direction
and Fourier basis in the angular direction. These basis
elements are orthonormal on the unit disk. Each input function to the convolution procedure is scaled to have support
on the disk of radius of one half and zero-padded on the unit disk. The result of the convolution (or correlation) then is a function which is supported on the unit disk. Since the convolution integral for compactly supported functions can be restricted from all of Euclidean space to the support of the functions,
it is only this integral over the support which is performed when using Fourier-Zernike expansions. Hence, the behavior of these functions outside of disks becomes irrelevant to the final result.
We work out how the Fourier-Zernike coefficients of the original functions appear in the convolution.

This article contains 4 sections. Section
2 is devoted to fixing notation and gives a
brief summary of convolution of functions on $\mathbb{R}^2$
and polar Fourier analysis. In Section
3, we present analytic aspects of the general theory of Fourier-Zernike series
for functions defined on disks. Section 4 is dedicated to study the presented theory of Fourier-Zernike series for convolution of functions supported on disks. As the main result we present a constructive closed form for Fourier-Zernike coefficients of convolution functions supported on disks. 
We then employ this closed form to present a constructive Fourier-Zernike approximation for convolution of zero-padded functions on $\mathbb{R}^2$.  

\section{{\bf Preliminaries and Notations}}

Throughout this section we shall present preliminaries and the notation. 

\subsection{General Notations.} 
For $d\in\mathbb{N}$ and $a>0$, let $\mathbb{B}_a^d:=\{\mathbf{x}\in\mathbb{R}^d:\|\mathbf{x}\|_2\le a\}$, where 
\[
\|\mathbf{x}\|_2:=\left(\sum_{\ell=1}^d|x_\ell|^2\right)^{1/2},
\]
for $\mathbf{x}:=(x_1,\cdots,x_d)^T\in\mathbb{R}^d$. We then put $\mathbb{B}^d:=\mathbb{B}_1^d$, that is the unit ball in $\mathbb{R}^d$.

It should be mentioned that, each function $f\in L^1(\mathbb{R}^d)$, satisfies the following integral decomposition;
\begin{equation}\label{int.dec}
\int_{\mathbb{R}^d}f(\mathbf{x})d\mathbf{x}=\int_{\mathbb{S}^{d-1}}\int_0^\infty f(r\mathbf{u})r^{d-1}drd\mathbf{u}.
\end{equation}
Also, if $f\in L^1(\mathbb{R}^d)$ is supported in $\mathbb{B}_a^d$, we then have 
\[
\int_{\mathbb{R}^d}f(\mathbf{x})d\mathbf{x}=\int_{\mathbb{S}^{d-1}}\int_0^af(r\mathbf{u})r^{d-1}drd\mathbf{u}.
\]
Let $d\in\mathbb{N}$, $a>0$ and $b:=a/2$. Let $f_1,f_2\in L^2(\mathbb{R}^d)$ with $\mathrm{supp}(f_1),\mathrm{supp}(f_2)\subseteq\mathbb{B}_{b}^d$. Then, we have 
$$\mathrm{supp}(f_1\ast f_2)\subseteq\mathrm{supp}(f_1)+\mathrm{supp}(f_2)\subseteq\mathbb{B}_{b}^d+\mathbb{B}_{b}^d\subseteq\mathbb{B}^d_a,$$ 
where 
\begin{equation}\label{conv.Rd}
f_1\ast f_2(\mathbf{x}):=\int_{\mathbb{R}^d}f_1(\mathbf{y})f_2(\mathbf{x}-\mathbf{y})d\mathbf{y},
\end{equation}
for $\mathbf{x}\in\mathbb{R}^d$.

Let $d\in\mathbb{N}$ and $C$ be a convex and compact set in $\mathbb{R}^d$.
Let $\mathrm{f}:C\to\mathbb{C}$ be a continuous function. 
Then, there exist a canonical extension of $f$ from $C$ to $\mathbb{R}^d$ by zero-padding, still denoted by $f:\mathbb{R}^d\to\mathbb{C}$ such that $f(x)=\mathrm{f}(x)$ for all $x\in C$, and $f(x)=0$ for all $x\not\in C$.  

Let $\mathrm{f}_k:C\to\mathbb{C}$ with $k\in\{1,2\}$ be continuous functions. We then define the canonical windowed convolution of $\mathrm{f}_1$ with $\mathrm{f}_2$, denoted by 
$\mathrm{f}_1\circledast\mathrm{f}_2$, by 
\begin{equation}
\mathrm{f}_1\circledast\mathrm{f}_2(\mathbf{x}):=f_1\ast f_2(\mathbf{x})=\int_{\mathbb{R}^d}f_1(\mathbf{y})f_2(\mathbf{x}-\mathbf{y})d\mathbf{y},
\end{equation} 
where $f_k$ is the canonical extension of $\mathrm{f}_k$ from $C$ to $\mathbb{R}^d$. We may also denote $\mathrm{f}_1\circledast\mathrm{f}_2$ by $f_1\circledast f_2$ as well. 

Since each $f_k$ is supported in $C$, we deduce that $f_1\ast f_2$ is supported in $C+C$. Hence, we get 
\begin{equation}
\mathrm{f}_1\circledast\mathrm{f}_2(\mathbf{x})=\int_{C}f_1(\mathbf{y})f_2(\mathbf{x}-\mathbf{y})d\mathbf{y}=\int_{C}\mathrm{f}_1(\mathbf{y})f_2(\mathbf{x}-\mathbf{y})d\mathbf{y},
\end{equation}
for all $\mathbf{x}\in C+C$.

Let $a>0$ and $b:=a/2$. Also, let $C:=\mathbb{B}_{b}^d$. Then, $C$ is a convex and compact set in $\mathbb{R}^d$. Also, we have $C+C\subseteq\mathbb{B}^d_a$.
Then, for continuous functions $\mathrm{f}_k:\mathbb{B}_{b}^d\to\mathbb{C}$ with $k\in\{1,2\}$, the convolution $f_1\ast f_2$ is supported in $\mathbb{B}^d_a$. Hence, we can write 
\begin{equation}
\mathrm{f}_1\circledast\mathrm{f}_2(\mathbf{x})=\int_{\mathbb{B}_{b}^d}\mathrm{f}_1(\mathbf{y})f_2(\mathbf{x}-\mathbf{y})d\mathbf{y},
\end{equation}
for all $\mathbf{x}\in\mathbb{B}^d_a$. Then, using the formula (\ref{int.dec}), we get 
\begin{equation}
\mathrm{f}_1\circledast\mathrm{f}_2(\mathbf{x})=\int_{\mathbb{S}^{d-1}}\int_{0}^{a/2}\mathrm{f}_1(r\mathbf{u})f_2(\mathbf{x}-r\mathbf{u})r^{d-1}drd\mathbf{u},
\end{equation}
for all $\mathbf{x}\in\mathbb{B}^d$.

\subsubsection{\bf The Case $d=2$.}
In this case, each function $f\in L^1(\mathbb{R}^2)$, satisfies the following integral decomposition;
\begin{equation}\label{int.dec2}
\int_{\mathbb{R}^2}f(\mathbf{x})d\mathbf{x}=\frac{1}{2\pi}\int_{0}^{2\pi}\int_0^\infty f(r,\theta)rdrd\theta.
\end{equation}
Also, if $f\in L^1(\mathbb{R}^2)$ is supported in $\mathbb{B}_a^2$, we then have 
\[
\int_{\mathbb{R}^2}f(\mathbf{x})d\mathbf{x}=\frac{1}{2\pi}\int_{0}^{2\pi}\int_0^af(r,\theta)rdrd\theta.
\]

\subsection{\bf Fourier-Zernike Analysis}
The radial Zernike function $Z_{nm}$, where $m\in\mathbb{Z}$ and $n\ge 0$ is an integer with $n\ge |m|$ and $n-|m|$ even, is a polynomial in $r$ given by \cite{13bhatia, F.Zernike}
\begin{equation}
Z_{nm}(r):=\sum_{j=0}^{\frac{n-|m|}{2}}(-1)^j\frac{(n-j)!}{j!(\frac{n+|m|}{2}-j)!(\frac{n-|m|}{2}-j)}r^{n-2j}.
\end{equation}
It has $(n-|m|)/2$ zeros between $0$ and $1$.

Also, for each $m\in\mathbb{Z}$ and $|m|\le n$ with $|m|\stackrel{2}{\equiv}n$, we have  
\begin{equation}\label{ZJ.Eq.1}
\int_0^1Z_{nm}(r)J_m(\alpha r)rdr=(-1)^{\frac{n-m}{2}}\cdot\frac{J_{n+1}(\alpha)}{\alpha},
\end{equation}
for each $0\not=\alpha\in\mathbb{R}$.

For a fixed $m\in\mathbb{Z}$, we have the following orthogonality relation;  
\begin{equation}
\int_0^1Z_{n_1m}(r)Z_{n_2m}(r)rdr=\frac{1}{2n_1+2}\cdot\delta_{n_1,n_2},
\end{equation}
for integer $n_j$, $j\in\{1,2\}$ with $n_j\ge |m|$ and $n_j-|m|$ even.

Therefore, for each $a>0$ and $m\in\mathbb{Z}$, we conclude  
\begin{equation}
\int_0^aZ_{n_1m}(a^{-1}r)Z_{n_2m}(a^{-1}r)rdr=\frac{a^2}{2n_1+2}\cdot\delta_{n_1,n_2},
\end{equation}
for integer $n_j$, $j\in\{1,2\}$ with $n_j\ge |m|$ and $n_j-|m|$ even.

Then, for a given $a>0$ and each $m\in\mathbb{Z}$, the set 
\begin{equation}
\mathcal{E}_m^a:=\left\{\mathcal{Z}_{nm}^a:n\ge |m|\ge 0\ {\rm and}\ |m|\stackrel{2}{\equiv}n\right\},
\end{equation} 
forms an orthonormal basis for the Hilbert function space $L^2([0,a],rdr)$, where 
\[
\mathcal{Z}_{nm}^a(r):=\frac{\sqrt{2n+2}}{a}\cdot Z_{nm}(a^{-1}r).
\]
In details, for integer $n_j$, $j\in\{1,2\}$ with $n_j\ge |m|$ and $n_j-|m|$ even, we have 
\[
\int_0^a\mathcal{Z}_{n_1m}^a(r)\mathcal{Z}_{n_2m}^a(r)rdr=\delta_{n_1,n_2}.
\]
Also, for each $m\in\mathbb{Z}$ and $|m|\le n$ with $|m|\stackrel{2}{\equiv}n$, we have  
\begin{equation}\label{ZJ.Eq.a}
\int_0^aZ_{nm}(a^{-1}r)J_m(\alpha r)rdr=a(-1)^{\frac{n-m}{2}}\cdot\frac{J_{n+1}(a\alpha)}{\alpha},
\end{equation}
for each $a>0$ and $0\not=\alpha\in\mathbb{R}$, where $J_{q}$ is the $q$-th order Bessel function of the first kind, for each $q\in\mathbb{Z}$.

Hence, any function $v:[0,a]\to\mathbb{R}$ satisfies the following expansion;
\begin{equation}\label{P.ONB.Z.a}
v(r)=\sum_{\{n:|m|\le n\ {\rm and}\ m\stackrel{2}{\equiv}n\}}\left(\int_0^av(s)\mathcal{Z}_{nm}^a(s)sds\right)\mathcal{Z}_{nm}^a(r),
\end{equation}
for $r\in[0,a]$.

We then can define the Fourier-Zernike basis element $V_{nm}^a$ in the polar form as follows;

\begin{equation}\label{2D.B.FrZ}
V_{nm}^a(r,\theta):=\frac{\sqrt{2n+2}}{a}\cdot Z_{nm}(a^{-1}r)\cdot\mathcal{Y}_m(\theta)=a^{-1}\sqrt{\frac{n+1}{\pi}}\cdot Z_{nm}(a^{-1}r)\cdot\exp(\ii m\theta),
\end{equation}
for $m\in\mathbb{Z}$ and $n\ge 0$ is an integer with $n\ge |m|$ and $n-|m|$ even.

Then, any restricted 2D integrable function $f(r,\theta)$ defined on $r\le a$ can be expanded with respect to $V_{nm}^a$ as defined in (\ref{2D.B.FrZ}) via
\begin{equation}\label{2D.EX.FrZa}
f(r,\theta)=\sum_{m=-\infty}^\infty\sum_{\{n:|m|\le n\ {\rm and}\ |m|\stackrel{2}{\equiv}n\}} C_{n,m}^a(f) V_{nm}^a(r,\theta),
\end{equation}
where 
\begin{equation}\label{2D.EX.FrZ.Coa}
C_{n,m}^a(f):=\int_0^a\int_0^{2\pi}f(r,\theta)\overline{V_{nm}^a(r,\theta)}rdr d\theta.
\end{equation}

\subsubsection{\bf The Case $a=1$}
In this case, any integrable function $v:[0,1]\to\mathbb{R}$ satisfies the following expansion;
\begin{equation}\label{P.ONB.1Z}
v(r)=\sum_{\{n:|m|\le n\ {\rm and}\ m\stackrel{2}{\equiv}n\}}\left(\int_0^1v(s)\mathcal{Z}_{nm}^1(s)sds\right)\mathcal{Z}_{nm}^1(r),
\end{equation}
for $r\in[0,1]$, where 
\[
Z_{nm}^1(r)=\sqrt{2n+2}Z_{nm}(r).
\]

Also, Fourier-Zernike basis elements $V_{nm}^1$ in the polar form have the following form;

\begin{equation}\label{2D.B.FrZ.1}
V_{nm}^1(r,\theta):=\sqrt{2n+2}Z_{nm}(r)\cdot\mathcal{Y}_m(\theta)=\sqrt{\frac{n+1}{\pi}}\cdot Z_{nm}(r)\cdot\exp(\ii m\theta),
\end{equation}
for $m\in\mathbb{Z}$ and $n\ge 0$ is an integer with $n\ge |m|$ and $n-|m|$ even.

Hence, any restricted 2D integrable function $f(r,\theta)$ defined on $r\le 1$ can be expanded with respect to $V_{nm}^1$ as defined in (\ref{2D.B.FrZ}) via
\begin{equation}\label{2D.EX.FrZ}
f(r,\theta)=\sum_{m=-\infty}^\infty\sum_{\{n:|m|\le n\ {\rm and}\ |m|\stackrel{2}{\equiv}n\}} C_{n,m}^1(f) V_{nm}^1(r,\theta),
\end{equation}
where 
\begin{equation}\label{2D.EX.FrZ.Co}
C_{n,m}^1(f):=\int_0^1\int_0^{2\pi}f(r,\theta)\overline{V_{nm}^1(r,\theta)}rdr d\theta.
\end{equation}

\newpage
\section{\bf Fourier-Zernike Series of Functions Supported on Disks}

This section is dedicated to study analytical aspects of Fourier-Zernike series of functions supported on disks (2D balls). We shall present a unified method for computing the Fourier-Zernike  coefficients of functions supported on disks. 

First, we need some preliminaries results. 

\begin{proposition}
{\it Let $r,a>0$, $0\le s\le a$ and $0<\alpha,\theta\le 2\pi$. We then have 
\begin{equation}\label{JZ.ge}
e^{\ii rs\cos(\alpha-\theta)}=\sqrt{2\pi}\sum_{m=-\infty}^{\infty}\sum_{\{n:|m|\le n\ {\rm and}\ |m|\stackrel{2}{\equiv}n\}}\sqrt{2n+2}\frac{\ii^m(-1)^{\frac{n-m}{2}}J_{n+1}(ar)}{r}e^{-\ii m\alpha}V_{nm}^a(s,\theta).
\end{equation}
}\end{proposition}
\begin{proof}
Let $\mathbf{x}=s\mathbf{u}_\theta$ and $\bom=r\mathbf{u}_\alpha$. By the Jacobi-Anger expansion, we can write 
\begin{equation}\label{JZ.e}
e^{\ii rs\cos(\alpha-\theta)}=e^{\ii {\bom}\cdot\mathbf{x}}=\sum_{m=-\infty}^{\infty}\ii^m J_m(rs)e^{\ii m\theta}e^{-\ii m\alpha}.
\end{equation}
Let $m\in\mathbb{Z}$. Expanding $J_m(rs)$ with respect to $s$, using (\ref{P.ONB.Z.a}), we can write 
\begin{align*}
J_m(rs)=\sum_{\{n:|m|\le n\ {\rm and}\ |m|\stackrel{2}{\equiv}n\}}\left(\int_0^a\mathcal{Z}_{nm}^a(p)J_{m}(rp)pdp\right)\mathcal{Z}_{nm}^a(s).
\end{align*}
Using (\ref{ZJ.Eq.a}), we can write 
\begin{align*}
\int_0^a\mathcal{Z}_{nm}^a(p)J_{m}(rp)pdp
&=\frac{\sqrt{2n+2}}{a}\int_0^aZ_{nm}(a^{-1}p)J_m(rp)pdp
\\&=\frac{\sqrt{2n+2}}{r}(-1)^{\frac{n-m}{2}}J_{n+1}(ar).
\end{align*}
We then deduce that 
\begin{equation}\label{JZ.ge.alt0}
J_m(rs)=r^{-1}\cdot\sum_{\{n:|m|\le n\ {\rm and}\ |m|\stackrel{2}{\equiv}n\}}\sqrt{2n+2}(-1)^{\frac{n-m}{2}}J_{n+1}(ar)\mathcal{Z}_{nm}^a(s).
\end{equation}
Applying Equation (\ref{JZ.ge.alt0}) in (\ref{JZ.e}), we get 
\begin{align*}
e^{\ii rs\cos(\alpha-\theta)}
&=\sum_{m=-\infty}^{\infty}\ii^m J_m(rs)e^{\ii m\theta}e^{-\ii m\alpha}
\\&=r^{-1}\cdot\sum_{m=-\infty}^{\infty}\ii^m\left(\sum_{\{n:|m|\le n\ {\rm and}\ |m|\stackrel{2}{\equiv}n\}}\sqrt{2n+2}(-1)^{\frac{n-m}{2}}J_{n+1}(ar)\mathcal{Z}_{nm}^a(s)\right)e^{\ii m\theta}e^{-\ii m\alpha}
\\&=\sqrt{2\pi}\sum_{m=-\infty}^{\infty}\sum_{\{n:|m|\le n\ {\rm and}\ |m|\stackrel{2}{\equiv}n\}}\sqrt{2n+2}\frac{\ii^m(-1)^{\frac{n-m}{2}}J_{n+1}(ar)}{r}e^{-\ii m\alpha}
V_{nm}^a(s,\theta).
\end{align*}
\end{proof}
We then conclude the following consequences. 

\begin{corollary}
{\it Let $a>0$ and $m\in\mathbb{Z}$, $n\ge |m|$ with $n\stackrel{2}{\equiv}|m|$. 
We then have 
\begin{equation}\label{main.altZ}
\int_{0}^a\int_0^{2\pi} e^{\ii rs\cos(\alpha-\theta)}\overline{V_{nm}^a(s,\theta)}sds d\theta=2\sqrt{\pi(n+1)}\frac{\ii^m(-1)^{\frac{n-m}{2}}J_{n+1}(ar)}{r}e^{-\ii m\alpha},
\end{equation}
for all $r>0$ and $0<\alpha\le 2\pi$.
}\end{corollary}

For an integral vector $\mathbf{k}:=(k_1,k_2)^T\in\mathbb{Z}^2$, let 
\[
\rho(\mathbf{k})=\rho(k_1,k_2):=\sqrt{k_1^2+k_2^2},
\] 
and 
$0\le \Phi(\mathbf{k})=\Phi(k_1,k_2)<2\pi$ be given by 
\[
k_1=\rho(k_1,k_2)\cos\Phi(k_1,k_2),\hspace{1cm}k_2=\rho(k_1,k_2)\sin\Phi(k_1,k_2).
\]
We may denote $\rho(\mathbf{k})$ with $|\mathbf{k}|$ as well.

\begin{corollary}
{\it Let $a>0$ and $\mathbf{k}\in\mathbb{Z}^2$. Also, let $m\in\mathbb{Z}$ and $n\ge |m|$ with $n\stackrel{2}{\equiv}|m|$. We then have 
\begin{equation}\label{main.alt.Zpolar}
\int_0^a\int_0^{2\pi}e^{\pi\ii a^{-1}s\mathbf{u}_\theta^T\mathbf{k}}\overline{V_{nm}^a(s\mathbf{u}_\theta)}sdsd\theta=2\sqrt{n+1}\frac{\ii^m(-1)^{\frac{n-m}{2}}J_{n+1}(\pi|\mathbf{k}|)}{a^{-1}\sqrt{\pi}|\mathbf{k}|}e^{-\ii m\Phi(\mathbf{k})}.
\end{equation}
}\end{corollary}
\begin{proof}
Let $a>0$ and $\mathbf{k}\in\mathbb{Z}^2$. 
Suppose $m\in\mathbb{Z}$ and $n\ge |m|$ with $n\stackrel{2}{\equiv}|m|$.
Applying Equation (\ref{main.altZ}), for $r:=a^{-1}\pi|\mathbf{k}|$ and $\alpha:=\Phi(\mathbf{k})$, we get 
\begin{align*}
\int_0^a\int_0^{2\pi}e^{\pi\ii a^{-1}s\mathbf{u}_\theta^T\mathbf{k}}\overline{V_{nm}^a(s\mathbf{u}_\theta)}sdsd\theta
&=\int_0^a\int_0^{2\pi}e^{\pi\ii a^{-1}s|\mathbf{k}|\cos(\Phi(\mathbf{k})-\theta)}\overline{V_{nm}^a(s\mathbf{u}_\theta)}sdsd\theta
\\&=\int_0^a\int_0^{2\pi}e^{\ii rs\cos(\alpha-\theta)}\overline{V_{nm}^a(s\mathbf{u}_\theta)}sdsd\theta
\\&=2\sqrt{\pi(n+1)}\frac{\ii^m(-1)^{\frac{n-m}{2}}J_{n+1}(ar)}{r}e^{-\ii m\alpha}
\\&=2\sqrt{\pi(n+1)}\frac{\ii^m(-1)^{\frac{n-m}{2}}J_{n+1}(\pi|\mathbf{k}|)}{a^{-1}\pi|\mathbf{k}|}e^{-\ii m\Phi(\mathbf{k})}
\\&=2\sqrt{n+1}\frac{\ii^m(-1)^{\frac{n-m}{2}}J_{n+1}(\pi|\mathbf{k}|)}{a^{-1}\sqrt{\pi}|\mathbf{k}|}e^{-\ii m\Phi(\mathbf{k})}.
\end{align*}
\end{proof}

Next result presents a closed form for Fourier-Zernike coefficients of functions defined on 
disks. 
\begin{theorem}\label{TH.C.Z.Oa}
Let $a>0$ and $\Omega_a:=[-a,a]^2$. Let $f\in L^2(\Omega_a)$ be a function supported 
in $\mathbb{B}^2_a$. Also, let $m\in\mathbb{Z}$ and $n\ge |m|$ with $n\stackrel{2}{\equiv}|m|$. 
We then have 
\begin{equation}\label{C.Z.Oa}
C_{n,m}^a(f)=\sum_{\mathbf{k}\in\mathbb{Z}^2}c_a(\mathbf{k};n,m)\widehat{f}(\mathbf{k}),
\end{equation} 
where, for each $\mathbf{k}\in\mathbb{Z}^2$, we have 
\begin{equation}
\widehat{f}(\mathbf{k}):=\int_{-a}^a\int_{-a}^af(x_1,x_2)e^{-\pi\ii a^{-1}(k_1x_1+k_2x_2)}dx_1dx_2,
\end{equation}
and 
\begin{equation}
c_a(\mathbf{k};n,m):=\sqrt{n+1}\frac{\ii^m(-1)^{\frac{n-m}{2}}J_{n+1}(\pi|\mathbf{k}|)}{2a\sqrt{\pi}|\mathbf{k}|}e^{-\ii m\Phi(\mathbf{k})}.
\end{equation}
\end{theorem}

\begin{proof}
Let $a>0$ and $\Omega_a:=[-a,a]^2$. Let $f\in L^2(\Omega_a)$ be a function supported 
in $\mathbb{B}^2_a$. Hence, we have   
\begin{equation}
f(\mathbf{x})=\frac{1}{4a^2}\sum_{\mathbf{k}\in\mathbb{Z}^2}\widehat{f}(\mathbf{k})e^{\pi\ii a^{-1} \mathbf{x}^T\mathbf{k}},
\end{equation}
for all $\mathbf{x}=(x_1,x_2)^T\in\Omega_a$, where for 
$\mathbf{k}=(k_1,k_2)^T\in\mathbb{Z}^2$, we have 
\[
\widehat{f}(\mathbf{k})=\int_{-a}^a\int_{-a}^af(x_1,x_2)e^{-\pi\ii a^{-1}(k_1x_1+k_2x_2)}dx_1dx_2.
\]
Hence, using (\ref{main.alt.Zpolar}), we get 
\begin{align*}
C_{n,m}^a(f)&=\int_0^a\int_0^{2\pi}f(s\mathbf{u}_\theta)\overline{V_{nm}^a(s\mathbf{u}_\theta)}sdsd\theta
\\&=\int_0^a\int_0^{2\pi}\left(\frac{1}{4a^2}\sum_{\mathbf{k}\in\mathbb{Z}^2}\widehat{f}(\mathbf{k})e^{\pi\ii a^{-1}s\mathbf{u}_\theta^T\mathbf{k}}\right)\overline{V_{nm}^a(s\mathbf{u}_\theta)}sds d\theta
\\&=\frac{1}{4a^2}\sum_{\mathbf{k}\in\mathbb{Z}^2}\widehat{f}(\mathbf{k})\left(\int_0^a\int_0^{2\pi}e^{\pi\ii a^{-1}s\mathbf{u}_\theta^T\mathbf{k}}\overline{V_{nm}^a(s\mathbf{u}_\theta)}sdsd\theta\right)
=\sum_{\mathbf{k}\in\mathbb{Z}^2}c_a(\mathbf{k};n,m)\widehat{f}(\mathbf{k}).
\end{align*}
\end{proof}

\begin{corollary}
{\it Let $a>0$ and $\Omega_a:=[-a,a]^2$. Let $f\in L^2(\Omega_a)$ be a function supported 
in $\mathbb{B}^2_a$.
We then have 
\[
f(r,\theta)=\sum_{m=-\infty}^\infty\sum_{\{n:|m|\le n\ {\rm and}\ |m|\stackrel{2}{\equiv}n\}} C_{n,m}^a(f) V_{nm}^a(r,\theta),
\]
for $0\le r\le a$ and $0\le \theta\le 2\pi$.
}\end{corollary}

\begin{remark}
The equation (\ref{C.Z.Oa}) guarantees that the Fourier-Zernike coefficients of functions supported in disks can be computed from the standard Fourier coefficients $\widehat{f}(\mathbf{k})$, which can be implemented by FFT. 
\end{remark}

Next result presents a closed form for Fourier-Zernike coefficients of functions supported in 
disks. 

\begin{theorem}\label{TH.C.Z.R2}
Let $a>0$ and $f\in L^1(\mathbb{R}^2)$ be a function supported in $\mathbb{B}_a^2$.
We then have 
\[
f(r,\theta)=\sum_{m=-\infty}^\infty\sum_{\{n:|m|\le n\ {\rm and}\ |m|\stackrel{2}{\equiv}n\}} C_{n,m}^a(f) V_{nm}^a(r,\theta),
\]
for $0\le r\le a$ and $0\le \theta\le 2\pi$, with
\begin{equation}\label{C.Z.R2}
C_{n,m}^a(f)=\sum_{\mathbf{k}\in\mathbb{Z}^2}c_a(\mathbf{k};n,m)\widehat{f}(\mathbf{k}),
\end{equation} 
where, for each $\mathbf{k}\in\mathbb{Z}^2$, we have 
\begin{equation}
\widehat{f}(\mathbf{k}):=\int_{-a}^a\int_{-a}^af(x_1,x_2)e^{-\pi\ii a^{-1}(k_1x_1+k_2x_2)}dx_1dx_2.
\end{equation}
\end{theorem}

Next result gives an explicit closed form for Fourier-Zernike coefficients of zero-padded functions. 

\begin{proposition}
{\it Let $a>0$ and $f\in L^1(\mathbb{R}^2)$ be a continuous function. Let $R(f)$
be the restriction of $f$ to the disk $\mathbb{B}_a^2$ and $E(f)$ be the extension of 
$R(f)$ to the rectangle $\Omega_a:=[-a,a]^2$ by zero-padding.
Also, let $m\in\mathbb{Z}$ and $n\ge |m|$ with $n\stackrel{2}{\equiv}|m|$. We then have 
\begin{equation}
C_{n,m}^a(R(f))=\sum_{\mathbf{k}\in\mathbb{Z}^2}c_a(\mathbf{k};n,m)\widehat{E(f)}(\mathbf{k}),
\end{equation} 
where, for $\mathbf{k}:=(k_1,k_2)^T\in\mathbb{Z}^2$;  
\begin{equation}
\widehat{E(f)}(\mathbf{k})=\int_{-a}^a\int_{-a}^af(x_1,x_2)e^{-\pi\ii a^{-1}(k_1x_1+k_2x_2)}dx_1dx_2.
\end{equation}
}\end{proposition}

Let $\mathcal{R}:=\left\{\rho(k_1,k_2):k_1,k_2\in\mathbb{Z}\right\}$. 
For each $r\in\mathcal{R}$, let 
\[
\Theta_r:=\left\{\Phi(i,j):r=\rho(i,j),~~i,j\in\mathbb{Z}\right\}.
\]

\begin{proposition} 
{\it With above assumptions we have
\begin{enumerate}
\item $\mathbb{N}\cup\{0\}\subseteq\mathcal{R}\subseteq\sqrt{\mathbb{N}}:=\{\sqrt{n}:n\in\mathbb{N}\cup\{0\}\}$.
\item $\mathcal{R}$ is a discrete subset of $[0,\infty)$.
\item For each $r\in\mathcal{R}$, the set $\Theta_r$ is a finite subset of $[0,2\pi)$.
\item $\mathbb{Z}^2=\bigcup_{r\in\mathcal{R}}\{(r\cos\theta,r\sin\theta):\theta\in\Theta_r\}$.
\end{enumerate}
}\end{proposition}
\begin{proof}
(1)-(3) are straightforward. 

(4) Let $\mathbf{x}\in\bigcup_{r\in\mathcal{R}}\{(r\cos\theta,r\sin\theta):\theta\in\Theta_r\}$. Suppose $r\in\mathcal{R}$ and $\theta\in\Theta_r$ with $\mathbf{x}=(r\cos\theta,r\sin\theta)$. Hence, $\theta=\Phi(i,j)$ with $\rho(i,j)=r$, for some $i,j\in\mathbb{Z}$. We then have 
\[
r\cos\theta=\rho(i,j)\cos\Phi(i,j)=\sqrt{i^2+j^2}\frac{i}{\sqrt{i^2+j^2}}=i\in\mathbb{Z},
\]
\[
r\sin\theta=\rho(i,j)\sin\Phi(i,j)=\sqrt{i^2+j^2}\frac{j}{\sqrt{i^2+j^2}}=j\in\mathbb{Z}.
\]
Thus, we deduce that $\mathbf{x}=(r\cos\theta,r\sin\theta)\in\mathbb{Z}^2$.
Therefore, we get $\bigcup_{r\in\mathcal{R}}\{(r\cos\theta,r\sin\theta):\theta\in\Theta_r\}\subseteq\mathbb{Z}^2$. Conversely, let $\mathbf{x}=(k_1,k_2)\in\mathbb{Z}^2$ be given.
We then have $k_1,k_2\in\mathbb{Z}$ and hence we get 
$k_1=\rho(k_1,k_2)\cos\Phi(k_1,k_2)$, and $k_2=\rho(k_1,k_2)\sin\Phi(k_1,k_2)$.
Then, we conclude that $\mathbf{x}=(r\cos\theta,r\sin\theta)$, 
with $r:=\rho(k_1,k_2)$ and $\theta:=\Phi(k_1,k_2)$. This implies that $\mathbf{x}\in\bigcup_{r\in\mathcal{R}}\{(r\cos\theta,r\sin\theta):\theta\in\Theta_r\}$ and hence 
$\mathbb{Z}^2\subseteq\bigcup_{r\in\mathcal{R}}\{(r\cos\theta,r\sin\theta):\theta\in\Theta_r\}$.
\end{proof}

We then present the following polarized version of Theorem \ref{TH.C.Z.Oa}.

\begin{theorem}
Let $a>0$ and $\Omega_a:=[-a,a]^2$. Let $f\in L^2(\Omega_a)$ be a function supported 
in $\mathbb{B}^2_a$. Also, let $m\in\mathbb{Z}$ and $n\ge |m|$ with $n\stackrel{2}{\equiv}|m|$. 
We then have 
\begin{equation}\label{C.ZBVC.Polar.Oa}
C_{n,m}^a(f)=\sum_{\tau\in\mathcal{R}}\sum_{\alpha\in\Phi_\tau}A_{mn}^a(\tau,\alpha)\widehat{f}(\tau\mathbf{u}_\alpha),
\end{equation}
where 
\begin{equation}
A_{mn}^a(\tau,\alpha):=c_a(\tau\mathbf{u}_\alpha;n,m).
\end{equation}
\end{theorem}
\begin{proof}
Let $m\in\mathbb{Z}$ and $n\ge |m|$ with $n\stackrel{2}{\equiv}|m|$. First, suppose that $\tau\in\mathcal{R}$ and $\alpha\in\Phi_\tau$. Let $\mathbf{k}:=\tau\mathbf{u}_{\alpha}=(\tau\cos\alpha,\tau\sin\alpha)^T\in\mathbb{Z}^2$. Thus, $|\mathbf{k}|=\tau$ and $\Phi(\mathbf{k})=\alpha$. 
Therefore, using (\ref{C.Z.Oa}), we get 
\begin{align*}
C_{n,m}(f)&=\sum_{\mathbf{k}\in\mathbb{Z}^2}c(\mathbf{k};n,m)\widehat{f}(\mathbf{k})
\\&=\sum_{\tau\in\mathcal{R}}\sum_{\alpha\in\Phi_\tau}c(\tau\mathbf{u}_\alpha;n,m)\widehat{f}(\tau\mathbf{u}_\alpha)
=\sum_{\tau\in\mathcal{R}}\sum_{\alpha\in\Phi_\tau}A_{mn}(\tau,\alpha)\widehat{f}(\tau\mathbf{u}_\alpha).
\end{align*}
\end{proof}

\begin{theorem}
Let $a>0$ and $f\in L^1(\mathbb{R}^2)$ be a function supported in $\mathbb{B}^2_a$.
Also, let $m\in\mathbb{Z}$ and $n\ge |m|$ with $n\stackrel{2}{\equiv}|m|$. We then have 
\begin{equation}\label{C.ZBVC.Polar}
C_{n,m}^a(f)=\sum_{\tau\in\mathcal{R}}\sum_{\alpha\in\Phi_\tau}A_{mn}^a(\tau,\alpha)\widehat{f}(\tau\mathbf{u}_\alpha),
\end{equation}
where 
\begin{equation}
A_{mn}^a(\tau,\alpha):=c_a(\tau\mathbf{u}_\alpha;n,m).
\end{equation}
\end{theorem}

Next result gives a polarized version for explicit closed form of Fourier-Zernike coefficients for zero-padded functions. 

\begin{proposition}
{\it Let $a>0$ and $f\in L^1(\mathbb{R}^2)$ be a continuous function. Let $R(f)$
be the restriction of $f$ to the unit disk $\mathbb{B}_a^2$ and $E(f)$ be the canonical extension of 
$R(f)$ to the rectangle $\Omega_a:=[-a,a]^2$ by zero-padding.
Also, let $m\in\mathbb{Z}$ and $n\ge |m|$ with $n\stackrel{2}{\equiv}|m|$. We then have 
\begin{equation}
C_{n,m}^a(R(f))=\sum_{\tau\in\mathcal{R}}\sum_{\alpha\in\Phi_\tau}A_{mn}^a(\tau,\alpha)\widehat{E(f)}(\tau\mathbf{u}_\alpha),
\end{equation} 
where 
\begin{equation}
\widehat{E(f)}(\tau\mathbf{u}_\alpha)=\int_{-a}^a\int_{-a}^af(x_1,x_2)e^{-\pi\ii\tau a^{-1}(x_1\cos\alpha+x_2\sin\alpha)}dx_1dx_2.
\end{equation}
}\end{proposition}

\begin{theorem}
Let $a>0$ and $f\in L^1(\mathbb{R}^2)$ be a continuous function. Let $R(f)$
be the restriction of $f$ to the unit disk $\mathbb{B}_a^2$ and $E(f)$ be the canonical extension of $R(f)$ to the rectangle $\Omega_a:=[-a,a]^2$ by zero-padding.
We then have 
We then have 
\[
f(r,\theta)=\sum_{m=-\infty}^\infty\sum_{\{n:|m|\le n\ {\rm and}\ |m|\stackrel{2}{\equiv}n\}} C_{n,m}^a(R(f))V_{nm}^a(r,\theta),
\]
for $0\le r\le a$ and $0\le \theta\le 2\pi$, where  
\begin{equation}
C_{n,m}^a(R(f))=\sum_{\tau\in\mathcal{R}}\sum_{\alpha\in\Phi_\tau}A_{mn}^a(\tau,\alpha)\widehat{E(f)}(\tau\mathbf{u}_\alpha).
\end{equation}
\end{theorem}

\section{\bf Fourier-Zernike Series for Convolution of Functions Supported on Disks}

We then continue by investigating analytical aspects of Fourier-Zernike 
series as a constructive approximation for convolution of functions supported on disks. 

The following theorem introduces a constructive method for computing the Fourier-Zernike coefficients of convolution of functions supported in disks.
\begin{theorem}\label{TH.conv.ZBC}
Let $a>0$ and $f_j\in L^1(\mathbb{R}^2)$ with $j\in\{1,2\}$ be functions supported in $\mathbb{B}^2_{a/2}$. Also, let $m\in\mathbb{Z}$ and $n\ge |m|$ with $n\stackrel{2}{\equiv}|m|$. The Fourier-Zernike coefficient $C_{n,m}^a(f_1\ast f_2)$ of $f_1\ast f_2$ is given by 
\begin{equation}\label{conv.ZBC}
C_{n,m}^a(f_1\ast f_2)=\sum_{\mathbf{k}\in\mathbb{Z}^2}c_a(\mathbf{k};n,m)\widehat{f_1}(\mathbf{k})\widehat{f_2}(\mathbf{k}),
\end{equation} 
where, for $j\in\{1,2\}$, $\mathbf{k}:=(k_1,k_2)^T\in\mathbb{Z}^2$, $m\in\mathbb{Z}$, and $n\in\mathbb{N}$, we have  
\begin{equation}\label{fj.hat}
\widehat{f_j}(\mathbf{k}):=\int_{-a}^a\int_{-a}^af_j(x_1,x_2)e^{-\pi\ii a^{-1}(k_1x_1+k_2x_2)}dx_1dx_2,
\end{equation}
and 
\begin{equation}\label{c.a.knm}
c_a(\mathbf{k};n,m):=\sqrt{n+1}\frac{\ii^m(-1)^{\frac{n-m}{2}}J_{n+1}(\pi|\mathbf{k}|)}{2a\sqrt{\pi}|\mathbf{k}|}e^{-\ii m\Phi(\mathbf{k})}.
\end{equation}
\end{theorem}
\begin{proof}
Let $a>0$ and $f_j\in L^1(\mathbb{R}^2)$ with $j\in\{1,2\}$ be  
functions supported in $\mathbb{B}^2_{a/2}$. 
Then, $f_1\ast f_2$ is supported in $\mathbb{B}_a^2$.
Let $m\in\mathbb{Z}$ and $n\ge |m|$ with $n\stackrel{2}{\equiv}|m|$. Regarding each $f_j$ as a function supported in $\mathbb{B}_a^2$, by the convolution property of Fourier transform, we have 
\begin{equation}\label{conv.f.prop}
\widehat{f_1\ast f_2}(\mathbf{k})=\widehat{f_1}(\mathbf{k})\widehat{f_2}(\mathbf{k}),
\end{equation}
for each $\mathbf{k}\in\mathbb{Z}^2$. Then, applying (\ref{conv.f.prop}) in Equation (\ref{C.Z.Oa}), we get 
\begin{align*}
C_{n,m}^a(f_1\ast f_2)&=\sum_{\mathbf{k}\in\mathbb{Z}^2}c_a(\mathbf{k};n,m)\widehat{f_1\ast f_2}(\mathbf{k})
\\&=\sum_{\mathbf{k}\in\mathbb{Z}^2}c_a(\mathbf{k};n,m)\widehat{f_1}(\mathbf{k})\widehat{f_2}(\mathbf{k}).
\end{align*}
\end{proof}

\begin{corollary}
{\it Let $a>0$ and $f_j\in L^1(\mathbb{R}^2)$ with $j\in\{1,2\}$ be functions supported in $\mathbb{B}^2_{a/2}$. We then have 
\begin{equation}\label{2D.EX.Fr.ast}
f_1\ast f_2(r,\theta)=\sum_{m=-\infty}^\infty\sum_{\{n:|m|\le n\ {\rm and}\ |m|\stackrel{2}{\equiv}n\}}C_{n,m}^{a}(f_1\ast f_2)V_{nm}^a(r,\theta),
\end{equation}
where 
\begin{equation}
C_{n,m}^a(f_1\ast f_2)=\sum_{\mathbf{k}\in\mathbb{Z}^2}c_a(\mathbf{k};n,m)\widehat{f_1}(\mathbf{k})\widehat{f_2}(\mathbf{k}).
\end{equation} 
}\end{corollary}

We then present the following polarized version of Theorem \ref{TH.conv.ZBC}.

\begin{theorem}
Let $a>0$ and $f_j\in L^1(\mathbb{R}^2)$ with $j\in\{1,2\}$ be 
functions supported in $\mathbb{B}^2_{a/2}$. Also, $m\in\mathbb{Z}$ and $n\ge |m|$ with $n\stackrel{2}{\equiv}|m|$. We then have 
\begin{equation}\label{conv.Z.Polar}
C_{n,m}^a(f_1\ast f_2)=\sum_{\tau\in\mathcal{R}}\sum_{\alpha\in\Phi_\tau}A_{mn}^a(\tau,\alpha)\widehat{f_1}(\tau\mathbf{u}_\alpha)\widehat{f_2}(\tau\mathbf{u}_\alpha).
\end{equation}
\end{theorem}

\begin{corollary}
{\it Let $a>0$ and $f_j\in L^1(\mathbb{R}^2)$ with $j\in\{1,2\}$ be functions supported in $\mathbb{B}^2_{a/2}$. We then have 
\begin{equation}
f_1\ast f_2(r,\theta)=\sum_{m=-\infty}^\infty\sum_{\{n:|m|\le n\ {\rm and}\ |m|\stackrel{2}{\equiv}n\}}C_{n,m}^{a}(f_1\ast f_2)V_{nm}^a(r,\theta),
\end{equation}
where 
\begin{equation}
C_{n,m}^a(f_1\ast f_2)=\sum_{\tau\in\mathcal{R}}\sum_{\alpha\in\Phi_\tau}A_{mn}^a(\tau,\alpha)\widehat{f_1}(\tau\mathbf{u}_\alpha)\widehat{f_2}(\tau\mathbf{u}_\alpha).
\end{equation} 
}\end{corollary}

Next we present a closed form for Fourier-Zernike coefficients of zero-padded functions.
 
\begin{theorem}
Let $a>0$ and $b:=a/2$. Suppose $f_j\in L^1(\mathbb{R}^2)$ with $j\in\{1,2\}$ be continuous functions. 
Let $R(f_j)$
be the restriction of $f_j$ to $\mathbb{B}_b^2$ and $E(f_j)$ be the canonical extension of 
$R(f_j)$ to the rectangle $\Omega_a:=[-a,a]^2$ by zero-padding.
Also, let $m\in\mathbb{Z}$ and $n\ge |m|$ with $n\stackrel{2}{\equiv}|m|$. We then have 
\begin{equation}\label{conv.zp}
C_{n,m}^a(f_1\circledast f_2)=\sum_{\mathbf{k}\in\mathbb{Z}^2}c_a(\mathbf{k};n,m)\widehat{f_1}[\mathbf{k}]\widehat{f_2}[\mathbf{k}],
\end{equation} 
where, for $\mathbf{k}=(k_1,k_2)^T\in\mathbb{Z}^2$;  
\begin{equation}
\widehat{f_j}[\mathbf{k}]:=\int_0^b\int_0^{2\pi}f_j(r,\theta)e^{-\pi\ii a^{-1}r(k_1\cos\theta+k_2\sin\theta)}rdrd\theta.
\end{equation}
\end{theorem}
\begin{proof}
Let $a>0$ and $b:=a/2$. Suppose $f_j\in L^1(\mathbb{R}^2)$ with $j\in\{1,2\}$ be continuous functions. Let $R(f_j)$ be the restriction of $f_j$ to the disk $\mathbb{B}_b^2$ and $E(f_j)$ be the extension of $R(f_j)$ to the rectangle $\Omega_a:=[-a,a]^2$ by zero-padding. We then have 
\[
f_1\circledast f_2=E(f_1)\ast E(f_2).
\]
Let $m\in\mathbb{Z}$ and $n\ge |m|$ with $n\stackrel{2}{\equiv}|m|$. Then, using Equation (\ref{conv.ZBC}), we get 
\begin{align*}
C_{n,m}^a(f_1\circledast f_2)&=C_{n,m}^a(E(f_1)\ast E(f_2))
\\&=\sum_{\mathbf{k}\in\mathbb{Z}^2}c_a(\mathbf{k};n,m)\widehat{E(f_1)}(\mathbf{k})\widehat{E(f_2)}(\mathbf{k}),
\end{align*}
with 
\[
\widehat{E(f_j)}(\mathbf{k})=\int_{-a}^a\int_{-a}^aE(f_j)(x_1,x_2)e^{-\pi\ii a^{-1}(k_1x_1+k_2x_2)}dx_1dx_2,
\]
for $j\in\{1,2\}$. Since each $E(f_j)$ is an extension of $R(f_j)$ to the rectangle $\Omega_a$ by zero padding, and $R(f_j)$ is restriction of $f_j$ to the disk $\mathbb{B}_b^2\subseteq[-b,b]^2$, we can write  
\begin{align*}
\widehat{E(f_j)}(\mathbf{k})
&=\int_{-a}^a\int_{-a}^aE(f_j)(x_1,x_2)e^{-\pi\ii a^{-1}(k_1x_1+k_2x_2)}dx_1dx_2
\\&=\int_{\Omega_a}E(f_j)(x_1,x_2)e^{-\pi\ii a^{-1}(k_1x_1+k_2x_2)}dx_1dx_2
\\&=\int_{\mathbb{B}_b^2}f_j(x_1,x_2)e^{-\pi\ii a^{-1}(k_1x_1+k_2x_2)}dx_1dx_2
\\&=\int_0^b\int_0^{2\pi}f_j(r,\theta)e^{-\pi\ii a^{-1}r(k_1\cos\theta+k_2\sin\theta)}rdrd\theta=\widehat{f_j}[\mathbf{k}].
\end{align*}
\end{proof}

\begin{corollary}
{\it Let $a>0$ and $b:=a/2$. Suppose $f_j\in L^1(\mathbb{R}^2)$ with $j\in\{1,2\}$ be continuous functions. Let $R(f_j)$ be the restriction of $f_j$ to $\mathbb{B}_b^2$ and $E(f_j)$ be the canonical extension of $R(f_j)$ to the rectangle $\Omega_a:=[-a,a]^2$ by zero-padding.
We then have 
\begin{equation}\label{2D.EX.Fr.ast.zerop}
f_1\circledast f_2(r,\theta)=\sum_{m=-\infty}^\infty\sum_{\{n:|m|\le n\ {\rm and}\ |m|\stackrel{2}{\equiv}n\}}C_{n,m}^{a}(f_1\circledast f_2)V_{nm}^a(r,\theta),
\end{equation}
where 
\begin{equation}
C_{n,m}^a(f_1\circledast f_2)=\sum_{\mathbf{k}\in\mathbb{Z}^2}c_a(\mathbf{k};n,m)\widehat{f_1}[\mathbf{k}]\widehat{f_2}[\mathbf{k}].
\end{equation} 
}\end{corollary}

We then present the following polarized version of closed forms for Fourier-Zernike approximations of zero-padded functions.

\begin{theorem}
Let $a>0$ and $b:=a/2$. Suppose $f_j\in L^1(\mathbb{R}^2)$ with $j\in\{1,2\}$ be continuous functions. Let $R(f_j)$ be the restriction of $f_j$ to $\mathbb{B}_b^2$ and $E(f_j)$ be the canonical extension of $R(f_j)$ to the rectangle $\Omega_a:=[-a,a]^2$ by zero-padding. Also, let $m\in\mathbb{Z}$ and $n\ge |m|$ with $n\stackrel{2}{\equiv}|m|$. We then have 
\begin{equation}
C_{n,m}^a(f_1\circledast f_2)=\sum_{\tau\in\mathcal{R}}\sum_{\alpha\in\Phi_\tau}A_{mn}^a(\tau,\alpha)\widehat{f_1}[\tau\mathbf{u}_\alpha]\widehat{f_2}[\tau\mathbf{u}_\alpha].
\end{equation}
\end{theorem}

\begin{corollary}
{\it Let $a>0$ and $b:=a/2$. Suppose $f_j\in L^1(\mathbb{R}^2)$ with $j\in\{1,2\}$ be continuous functions. Let $R(f_j)$ be the restriction of $f_j$ to $\mathbb{B}_b^2$ and $E(f_j)$ be the canonical extension of $R(f_j)$ to the rectangle $\Omega_a:=[-a,a]^2$ by zero-padding. We then have 
\begin{equation}
f_1\circledast f_2(r,\theta)=\sum_{m=-\infty}^\infty\sum_{\{n:|m|\le n\ {\rm and}\ |m|\stackrel{2}{\equiv}n\}}C_{n,m}^{a}(f_1\circledast f_2)V_{nm}^a(r,\theta),
\end{equation}
}\end{corollary}

\subsection{\bf Convolution approximation of Fourier-Zurnike basis elements}
Let $a>0$ and $b:=a/2$. Suppose $f_j:\mathbb{R}^2\to\mathbb{R}$ with $j\in\{1,2\}$ be continuous functions supported in the disk $\mathbb{B}_{b}^2$ with the associated Fourier-Zernike coefficients $\left\{C_{n,m}^b(f_j):m\in\mathbb{Z},\ n\in\mathbb{I}_m\right\}$, with $\mathbb{I}_m:=\{n:|m|\le n\ {\rm and}\ |m|\stackrel{2}{\equiv}n\}$. Hence, we can write 
\begin{equation}
f_j=\sum_{m=-\infty}^\infty\sum_{\{n:|m|\le n\ {\rm and}\ |m|\stackrel{2}{\equiv}n\}} 
C_{n,m}^b(f_j)V_{nm}^b,
\end{equation}
where 
\[
C_{n,m}^b(f_j)=\sum_{\mathbf{k}\in\mathbb{Z}^2}c_b(\mathbf{k};n,m)\widehat{f_j}(\mathbf{k}),
\]
for $m\in\mathbb{Z}$ and $n\ge |m|$ with $n\stackrel{2}{\equiv}|m|$.

Using linearity of convolutions, as linear operators, we get  
\begin{align*}
f_1\ast f_2 &=\left(\sum_{m=-\infty}^\infty\sum_{\{n:|m|\le n\ {\rm and}\ |m|\stackrel{2}{\equiv}n\}}C_{n,m}^b(f_1)V_{nm}^b\right)\ast\left(\sum_{m'=-\infty}^\infty\sum_{\{n':|m'|\le n'\ {\rm and}\ |m'|\stackrel{2}{\equiv}n'\}} C_{n',m'}^b(f_2)V_{n'm'}^b\right)
\\&=\sum_{m=-\infty}^\infty\sum_{\{n:|m|\le n\ {\rm and}\ |m|\stackrel{2}{\equiv}n\}}\sum_{m'=-\infty}^\infty\sum_{\{n':|m'|\le n'\ {\rm and}\ |m'|\stackrel{2}{\equiv}n'\}} C_{n,m}^b(f_1)C_{n',m'}^b(f_2) V_{nm}^b\ast V_{n'm'}^b,
\end{align*}
where $V_{nm}^b\ast V_{n'm'}^b$ is the standard convolution of Fourier-Zernike basis elements, considering them as functions defined on $\mathbb{R}^2$ by zero-padding and supported 
$\mathbb{B}_b^2$.
 
Therefore, convolution of Fourier-Zernike basis elements can be viewed as pre-computed kernels. 

\begin{proposition}
{\it Let $a>0$ and $b:=a/2$. Suppose $\mathbf{k}\in\mathbb{Z}^2$, $m\in\mathbb{Z}$ and $n\ge |m|$ with $n\stackrel{2}{\equiv}|m|$. We then have 
\begin{equation}
\widehat{V_{nm}^b}[\mathbf{k}]=2\sqrt{n+1}\ii^{-m}e^{\ii m\Phi(\mathbf{k})}a(-1)^{n-m}\frac{J_{n+1}(\pi|\mathbf{k}|/2)}{\sqrt{\pi}|\mathbf{k}|}.
\end{equation}
}\end{proposition}
\begin{proof}
Let $m\in\mathbb{Z}$ and $n\ge |m|$ with $n\stackrel{2}{\equiv}|m|$. 
Regarding $V_{nm}^b$ as a function defined on $\mathbb{R}^2$ by zero-padding and supported in $\mathbb{B}_{b}^2$, still denoted by $V_{nm}^b$, for 
each $\mathbf{k}\in\mathbb{Z}^2$, we can write 
\begin{align*}
\widehat{V_{nm}^b}[\mathbf{k}]
&=\widehat{V_{nm}^b}[|\mathbf{k}|\mathbf{u}_{\Phi(\mathbf{k})}]
\\&=\int_{\mathbb{B}^2_b}V_{nm}(\mathbf{x})e^{-\pi\ii |\mathbf{k}|\mathbf{u}_{\Phi(\mathbf{k})}^T\mathbf{x}}d\mathbf{x}
\\&=\int_0^{2\pi}\int_0^{b}V_{nm}^b(r,\theta)e^{-\pi\ii a^{-1}r|\mathbf{k}|\mathbf{u}_{\Phi(\mathbf{k})}^T\mathbf{u}_\theta}rdrd\theta
\\&=\int_0^{2\pi}\int_0^{b}V_{nm}^b(r,\theta)\left(\sum_{l=-\infty}^{\infty}\ii^{-l}
J_l(a^{-1}\pi r|\mathbf{k}|)e^{-\ii l\theta}e^{\ii l\Phi(\mathbf{k})}\right)rdrd\theta
\\&=\frac{\sqrt{2n+2}}{b\sqrt{2\pi}}\int_0^{2\pi}\int_0^{b} Z_{nm}(b^{-1}r)e^{\ii m\theta}\left(\sum_{l=-\infty}^{\infty}\ii^{-l}
J_l(a^{-1}\pi r|\mathbf{k}|)e^{-\ii l\theta}e^{\ii l\Phi(\mathbf{k})}\right)rdrd\theta
\\&=\frac{\sqrt{2n+2}}{b\sqrt{2\pi}}\sum_{l=-\infty}^{\infty}\ii^{-l}e^{\ii l\Phi(\mathbf{k})}
\left(\int_0^{2\pi}\int_0^{b} Z_{nm}(b^{-1}r)
J_l(a^{-1}\pi r|\mathbf{k}|)e^{\ii m\theta}e^{-\ii l\theta}rdrd\theta \right)
\\&=\frac{\sqrt{2n+2}}{b\sqrt{2\pi}}\sum_{l=-\infty}^{\infty}\ii^{-l}e^{\ii l\Phi(\mathbf{k})}
\left(\int_0^{2\pi}e^{\ii m\theta}e^{-\ii l\theta}d\theta\right)\left(\int_0^{b} Z_{nm}(b^{-1}r)J_l(a^{-1}\pi r|\mathbf{k}|)rdr\right)
\\&=\frac{\sqrt{2\pi(2n+2)}}{b}\sum_{l=-\infty}^{\infty}\ii^{-l}e^{\ii l\Phi(\mathbf{k})}\delta_{ml}\left(\int_0^{b} Z_{nm}(b^{-1}r)J_l(a^{-1}\pi r|\mathbf{k}|)rdr\right)
\\&=\frac{\sqrt{2\pi(2n+2)}}{b}\ii^{-m}e^{\ii m\Phi(\mathbf{k})}\left(
\int_0^{b}Z_{nm}(b^{-1}r)J_m(a^{-1}\pi r|\mathbf{k}|)rdr\right).
\end{align*}
Hence, we get 
\begin{equation}\label{main.drum}
\widehat{V_{nm}^b}[\mathbf{k}]=\frac{\sqrt{2\pi(2n+2)}}{b}\ii^{-m}e^{\ii m\Phi(\mathbf{k})}\left(
\int_0^{b}Z_{nm}(b^{-1}r)J_m(a^{-1}\pi r|\mathbf{k}|)rdr\right).
\end{equation}
Applying (\ref{ZJ.Eq.a}) in (\ref{main.drum}), we can write  
\begin{align*}
\int_0^{b}Z_{nm}(b^{-1}r)J_m(a^{-1}\pi r|\mathbf{k}|)rdr=b(-1)^{\frac{n-m}{2}}\frac{J_{n+1}(ba^{-1}\pi|\mathbf{k}|)}{a^{-1}\pi|\mathbf{k}|}
=a^{2}(-1)^{\frac{n-m}{2}}\frac{J_{n+1}(\pi|\mathbf{k}|/2)}{2\pi|\mathbf{k}|}.
\end{align*}
which implies that 
\begin{align*}
\widehat{V_{nm}^b}[\mathbf{k}]
&=\frac{\sqrt{2\pi(2n+2)}}{b}\ii^{-m}e^{\ii m\Phi(\mathbf{k})}\left(
\int_0^{b}Z_{nm}(b^{-1}r)J_m(a^{-1}\pi r|\mathbf{k}|)rdr\right)
\\&=\frac{\sqrt{2\pi(2n+2)}}{b}\ii^{-m}e^{\ii m\Phi(\mathbf{k})}a^{2}(-1)^{\frac{n-m}{2}}\frac{J_{n+1}(\pi|\mathbf{k}|/2)}{2\pi|\mathbf{k}|}
\\&=2\sqrt{n+1}\ii^{-m}e^{\ii m\Phi(\mathbf{k})}a(-1)^{n-m}\frac{J_{n+1}(\pi|\mathbf{k}|/2)}{\sqrt{\pi}|\mathbf{k}|}.
\end{align*}
\end{proof}

\begin{proposition}
{\it Let $a>0$ and $b:=a/2$. Also, let $m,m'\in\mathbb{Z}$ and $n\in\mathbb{I}_m$, $n'\in\mathbb{I}_{m'}$. 
Then, for each $\ell\in\mathbb{Z}$ and $k\in\mathbb{I}_\ell$, we have  
\begin{equation}
C_{k,\ell}^a(V_{nm}^{b}\circledast V_{n'm'}^{b})=\frac{a^2\sqrt{(n+1)(n'+1)}}{\ii^{(m+m')}(-1)^{m-n-m'+n'}\pi}\sum_{\mathbf{k}\in\mathbb{Z}^2}c_a(\mathbf{k};k,\ell)\frac{e^{\ii(m+m')\Phi(\mathbf{k})}J_{n'+1}(\pi|\mathbf{k}|/2)J_{n'+1}(\pi|\mathbf{k}|/2)}{|\mathbf{k}|},
\end{equation}
}\end{proposition}
\begin{proof}
Let $a>0$ and $b:=a/2$. Also, let $m,m'\in\mathbb{Z}$ and $n\in\mathbb{I}_m$, $n'\in\mathbb{I}_{m'}$. Regarding $V_{nm}^{b}$ and $V_{n'm'}^{b}$ as functions supported in 
$\mathbb{B}_{b}^2$, $V_{nm}^b\circledast V_{n'm'}^b$ is a function supported in the disk 
$\mathbb{B}^2_a$. Suppose $\ell\in\mathbb{Z}$ and $k\in\mathbb{I}_\ell$. 
Using (\ref{conv.zp}), we have  
\begin{align*}
C_{k,\ell}^a(V_{nm}^{b}\circledast V_{n'm'}^{b})
&=\sum_{\mathbf{k}\in\mathbb{Z}^2}c_a(\mathbf{k};k,\ell)\widehat{V_{nm}^{b}}[\mathbf{k}]\widehat{V_{n'm'}^{b}}[\mathbf{k}]
\\&=\frac{a^2\sqrt{(n+1)(n'+1)}}{\ii^{(m+m')}(-1)^{m-n-m'+n'}\pi}\sum_{\mathbf{k}\in\mathbb{Z}^2}c_a(\mathbf{k};k,\ell)\frac{e^{\ii(m+m')\Phi(\mathbf{k})}J_{n'+1}(\pi|\mathbf{k}|/2)J_{n'+1}(\pi|\mathbf{k}|/2)}{|\mathbf{k}|}.
\end{align*}
\end{proof}

\begin{theorem}
Let $a>0$ and $b:=a/2$. Suppose $m,m'\in\mathbb{Z}$, 
$n\in\mathbb{I}_m$ and $n'\in\mathbb{I}_{m'}$. We then have 
\begin{equation}
V_{nm}^{b}\circledast V_{n'm'}^{b}(r,\theta)=\sum_{\ell=\infty}^\infty\sum_{k\in\mathbb{I}_\ell}C_{k,\ell}^a(V_{nm}^{b}\circledast V_{n'm'}^{b})V_{k,\ell}^a(r,\theta),
\end{equation}
where for each $\ell\in\mathbb{Z}$ and $k\in\mathbb{I}_\ell$, we have 
\begin{equation}
C_{k,\ell}^a(V_{nm}^{b}\circledast V_{n'm'}^{b})=\frac{a^2\sqrt{(n+1)(n'+1)}}{\ii^{(m+m')}(-1)^{m-n-m'+n'}\pi}\sum_{\mathbf{k}\in\mathbb{Z}^2}c_a(\mathbf{k};k,\ell)\frac{e^{\ii(m+m')\Phi(\mathbf{k})}J_{n'+1}(\pi|\mathbf{k}|/2)J_{n'+1}(\pi|\mathbf{k}|/2)}{|\mathbf{k}|}.
\end{equation}
\end{theorem}

\

{\bf Concluding Remarks.}
The mathematical foundations for computing convolutions
of functions supported on disks in the plane are derived.  The motivation for this work is the way that the Fourier-Zernike basis transforms under rotation, which is not shared by the multi-dimensional Fourier series of
periodized functions.
Extensions to functions supported on balls in $d$-dimensional Euclidean space with the Fourier series for the angular direction being replaced by hyper-spherical harmonics follow in a natural way.

\

{\bf Acknowledgments.} 
This work has been supported by the National Institute of General Medical Sciences of the NIH under award number R01GM113240, by the US National Science Foundation under grant NSF CCF-1640970, and by Office of Naval Research Award N00014-17-1-2142. The authors gratefully acknowledge the supporting agencies. The findings and opinions expressed here are only those of the authors, and not of the funding agencies.

\bibliographystyle{amsplain}

\end{document}